\documentclass[a4paper,10pt]{article}
\usepackage{amsthm}
\usepackage[english]{babel}
\usepackage{inputenc}
\usepackage{float,url}
\usepackage{rotating}
\usepackage[T1]{fontenc}
\usepackage{amsmath}
\usepackage{amssymb}
\usepackage{amsfonts}
\usepackage{graphicx,subfigure,enumerate}

\newtheorem{cor}{Corollary}

\newtheorem{theo}{Theorem}
\newtheorem{lem}{Lemma}
\newtheorem{preuve*}{Proof}
\usepackage{latexsym}
\usepackage{color}
\usepackage{graphics}
\usepackage{enumitem}
\usepackage{amsmath,amsfonts,graphics,amssymb,vmargin}
\usepackage{epic,eepic,epsfig,epsf}
\usepackage{color}
\usepackage{pst-plot}
\usepackage{pstricks-add}
\usepackage{pst-eucl}
\usepackage{ulem}
\usepackage[final,xcolor=dvipdf]{changes}
\usepackage{listings}
\usepackage{array}

\graphicspath{{images/}}
\newcommand{\argmin}{\mathop{\mathrm{argmin}}}
\newcommand{\E}{\mathbb{E}}

\newcommand{\R}{\mathbb{R}}

\newcommand{\V}{\mathbb{V}}
\newcommand{\parenth}[1]{\left(#1\right)}
\newcommand{\Tr}{\mathrm{T}}

\newcommand{\bigO}[1]{\ensuremath{\mathop{}\mathopen{}O\mathopen{}\left(#1\right)}}
\newcommand{\dof}{\mathrm{dof}}
\DeclareMathOperator{\se}{SE}
\DeclareMathOperator{\rank}{rank}
\DeclareMathOperator{\mse}{Risk}
\DeclareMathOperator{\sure}{SURE}
\DeclareMathOperator{\supp}{supp}

\DeclareMathOperator{\sign}{sign}
\DeclareMathOperator{\divg}{div}
\DeclareMathOperator{\cov}{cov}
\renewcommand{\span}{\mathrm{span}}

\DeclareMathOperator{\ball}{Ball}
\DeclareMathOperator{\tr}{tr}

\date{}

\begin{document}
\title{\bf The degrees of freedom of the Lasso for general design matrix}
\author{C. Dossal$^{(1)}$ M. Kachour$^{(2)}$, M.J. Fadili$^{(2)}$, G. Peyr\'e$^{(3)}$ and C. Chesneau$^{(4)}$}
\maketitle

\noindent
\begin{tabular}{@{}c@{}@{}c@{}}
\begin{tabular}{>{\centering}p{0.48\textwidth}}
$(1)$ IMB, CNRS-Univ. Bordeaux 1\\
351 Cours de la Lib\'eration, F-33405 Talence, France\\
\url{Charles.Dossal@math.u-bordeaux1.fr}
\end{tabular} &
\begin{tabular}{>{\centering}p{0.48\textwidth}}
$(2)$ GREYC, CNRS-ENSICAEN-Univ. Caen\\
6 Bd du Mar\'echal Juin, 14050 Caen, France\\
\url{Jalal.Fadili}\url{@greyc.ensicaen.fr} \\
\url{Maher.Kachour@greyc.ensicaen.fr}
\end{tabular} \\\\
\begin{tabular}{>{\centering}p{0.48\textwidth}}
$(3)$ Ceremade, CNRS-Univ. Paris-Dauphine\\
Place du Mar\'echal De Lattre De Tassigny, 75775 Paris 16, France\\
\url{Gabriel.Peyre@ceremade.dauphine.fr}
\end{tabular} &
\begin{tabular}{>{\centering}p{0.48\textwidth}}
$(4)$ LMNO, CNRS-Univ. Caen\\
D\'epartement de Math\'ematiques, UFR de Sciences, 14032 Caen, France\\
\url{Chesneau.Christophe@math.unicaen.fr}
\end{tabular}
\end{tabular}

\begin{abstract}
In this paper, we investigate the degrees of freedom ($\dof$) of penalized $\ell_1$ minimization (also known as the Lasso) for linear regression models. We give a closed-form expression of the $\dof$ of the Lasso response. Namely, we show that for any given Lasso regularization parameter $\lambda$ and any observed data $y$ belonging to a set of full (Lebesgue) measure, the cardinality of the support of a particular solution of the Lasso problem is an unbiased estimator of the degrees of freedom. This is achieved without the need of uniqueness of the Lasso solution. Thus, our result holds true for both the underdetermined and the overdetermined case, where the latter was originally studied in \cite{zou}. \added{We also show, by providing a simple counterexample, that although the $\dof$ theorem of \cite{zou} is correct, their proof contains a flaw since their divergence formula holds on a different set of a full measure than the one that they claim. An effective estimator of the number of degrees of freedom may have several applications including an objectively guided choice of the regularization parameter in the Lasso through the $\sure$ framework. Our theoretical findings are illustrated through several numerical simulations.}
%Namely, we study the degrees of freedom of the Lasso response in the framework of $\sure$. We show that, for any given $\lambda$, outside a union of hyperplanes, the cardinality of the support of a particular solution of the Lasso problem is an unbiased estimator of the degrees of freedom of the Lasso response. This result remains true when the Lasso problem has a unique solution. Thus, we compare our result with [\cite{zou}, Theorem 1] for the overdetermined case. 
\end{abstract}
\paragraph{Keywords:} Lasso, model selection criteria, degrees of freedom, $\sure$.
\paragraph{AMS classification code:} Primary $62M10$, secondary $62M20$.

%\newpage
%\tableofcontents
\newpage

%%%%%%%%%%%%%%%%%%%%%%%%%%%%%%%%%%%%%%%%%%%%%%%%%%%%%%%%%%%%%%
\section{Introduction}\label{section1}
%%%%%%%%%%%%%%%%%%%%%%%%%%%%%%%%%%%%%%%%%%%%%%%%%%%%%%%%%%%%%%

%%%%%%%%%%%%%%%%%%%%%%%%%%%%%%
\subsection{Problem statement}\label{subsec:statement}
%Recently, statistical problems in high dimension, have become increasingly common. That is, estimation problems where the dimension $p$ of the parameter to be estimated is larger than the size $n$ of the sample.
%In this paper, we consider the following underdetermined linear regression model
We consider the following linear regression model
\begin{equation}\label{eq0}
y = A x^{0} + \varepsilon,\qquad \mu=Ax^0,
\end{equation}
where $y \in \R^{n}$ is the observed data or the response vector, $A = \left( a_{1}, \cdots, a_{p} \right)$ is an $n \times p$ design matrix, $x^{0} = \left( x^{0}_{1}, \cdots, x^{0}_{p} \right)^{\Tr}$ is the vector of unknown regression coefficients and $ \varepsilon $ is a vector of i.i.d. centered Gaussian random variables with variance $\sigma^{2} > 0$. In this paper, the number of observations $n$ can be greater than the ambient dimension $p$ of the regression vector to be estimated. Recall that when $n < p$, \eqref{eq0} is an underdetermined linear regression model, whereas when $n \geq p$ and all the columns of $A$ are linearly independent, it is overdetermined. 
%Recall that when $n < p$, \eqref{eq0} is called underdetermined linear regression model, which is probably the most famous example of statistical problems in high dimensional. On the other hand, when all the vectors of the design matrix $A$ are linearly independent, which is only possible if $n \geq p$, \eqref{eq0} is called overdetermined linear regression model.

\added{Let $\widehat{x}(y)$ be an estimator of $x^0$, and $\widehat{\mu}(y) = A \widehat{x}(y)$ be the associated response or predictor. The concept of degrees of freedom plays a pivotal role in quantifying the complexity of a statistical modeling procedure. More precisely, since $y \sim \mathcal{N}( \mu = A x^0, \sigma^{2} \mathrm{Id}_{n \times n} )$ ($\mathrm{Id}_{n \times n}$ is the identity on $\R^n$), according to \cite{efron}, the degrees of freedom ($\dof$) of the response $\widehat{\mu}(y)$ is defined by
\begin{equation}\label{eq2}
df = \sum_{i = 1}^{n} \dfrac{ \cov ( \widehat{\mu}_{i}(y) , y_{i} )}{\sigma^{2}}.
\end{equation}
%Tibshirani \cite{tibshirani} links the prediction risk (or the true prediction error) $\mse (\widehat{\mu}) = \E \Vert \widehat{\mu} - \mu \Vert_{2}^{2}$ to the degree of freedom $df ( \widehat{\mu} )$ of the estimator $\widehat{\mu}$:
%\begin{equation}\label{linkk}
%\mse (\widehat{\mu}) = \E \lbrace - n \sigma^{2} + \Vert \widehat{\mu} - y \Vert_{2}^{2} + 2 \sigma^{2} df (\widehat{\mu}) \rbrace,
%\end{equation}
%where $\Vert \cdot \Vert_{2}$ denotes the $\ell_{2}$ norm. 
Many model selection criteria involve $df$, e.g. $C_{p}$ (Mallows \cite{mallows}), $\mbox{AIC}$ (Akaike Information Criterion, \cite{akaike}), $\mbox{BIC}$ (Bayesian Information Citerion, \cite{schwarz}), $\mbox{GCV}$ (Generalized Cross Validation, \cite{craven}) and $\sure$ (Stein's unbiased risk estimation \cite{stein}, see Section~\ref{fiabilite}). Thus, the $\dof$ is a quantity of interest in model validation and selection and it can be used to get the optimal hyperparameters of the estimator. Note that the optimality here is intended in the sense of the prediction $\widehat{\mu}(y)$ and not the coefficients $\widehat{x}(y)$.} 

\added{The well-known Stein's lemma \cite{stein} states that if $y \mapsto \widehat{\mu}(y)$ is weakly differentiable then its divergence is an unbiased estimator of its degrees of freedom, i.e.
\begin{equation}\label{eq4}
\widehat{df}(y) = \divg(\widehat{\mu}(y)) =  \sum_{i = 1}^{n} \dfrac{\partial \widehat{\mu}_{i}(y)}{\partial y_{i}}, \quad \text{and} \quad \mathbb{E}(\widehat{df}(y)) = df ~.
\end{equation}
\\

Here, in order to estimate $x^{0}$, we consider solutions to the Lasso problem, proposed originally in \cite{tibshirani}. The Lasso amounts to solving the following convex optimization problem
\begin{equation}\label{eq1}\tag{$\mbox{P}_{1} ( y, \lambda )$}
\min_{x \in \R^{p}} \dfrac{1}{2} \Vert y - A x \Vert_{2}^{2} + \lambda \Vert x \Vert_{1},
\end{equation}
where $\lambda > 0$ is called the Lasso regularization parameter and $\Vert \cdot \Vert_{2}$ (resp. $\Vert \cdot \Vert_{1}$) denotes the $\ell_{2}$ (resp. $\ell_{1}$) norm. 
%The convexity of this minimization problem ensures that the estimator can be computed efficiently even if $n < p$ and with very large $p$. 
An important feature of the Lasso is that it promotes sparse solutions. In the last years, there has been a huge amount of work where efforts have focused on investigating the theoretical guarantees of the Lasso as a sparse recovery procedure from noisy measurements. See, e.g., \cite{fanetli,fan,zhao,zouu,ravi,nardi,osborne1,efronn,fuchs,tropp}, to name just a few.}

\subsection{Contributions and related work}
Let $\widehat{\mu}_{\lambda}(y) = A  \widehat{x}_{\lambda}(y)$ be the Lasso response vector, where $\widehat{x}_{\lambda} ( y )$ is a solution of the Lasso problem \eqref{eq1}. Note that all minimizers of the Lasso share the same image under $A$, i.e. $\widehat{\mu}_{\lambda}(y)$ is uniquely defined; see Lemma~\ref{lem3} in Section~\ref{preuves} for details. The main contribution of this paper is first to provide an unbiased estimator of the degrees of freedom of the Lasso response for any design matrix. The estimator is valid everywhere except on a set of (Lebesgue) measure zero. We reach our goal without any additional assumption to ensure uniqueness of the Lasso solution. Thus, our result covers the challenging underdetermined case where the Lasso problem does not necessarily have a unique solution. It obviously holds when the Lasso problem \eqref{eq1} has a unique solution, and in particular in the overdetermined case studied in \cite{zou}. \added{Using the estimator at hand, we also establish the reliability of the $\sure$ as an unbiased estimator of the Lasso prediction risk.}

%\deleted{Indeed, for the overdetermined case, authors \cite{zou} shows that for a fixed $\lambda$, and $y$ outside a finite union of hyperplanes, the number of non-zero coefficients of the unique solution of the Lasso problem is an unbiased estimator of the degrees of the freedom of the response Lasso. In this work, we arrive at a similar expression of the degrees of freedom as in \cite{zou} for the overdetermined case, but with the notable distinction that it holds on a different set (of full measure) for the observed data $y$.}

\added{While this paper was submitted, we became aware of the independent work of Tibshirani and Taylor \cite{tib12}, who studied the $\dof$ for general $A$ both for the Lasso and the general (analysis) Lasso. }

\added{Section~\ref{zou} is dedicated to a thorough comparison and discussion of connections and differences between our results and the one in \cite[Theorem 1]{zou} for the overdetermined case, and that of \cite{kato09,tib12,vaiter11} for the general case.}

\subsection{Overview of the paper}
This paper is organized as follows. Section~\ref{resultats} is the core contribution of this work where we state our main results. There, we provide the unbiased estimator of the $\dof$ of the Lasso, and we investigate the reliability of the $\sure$ estimate of the Lasso prediction risk. Then, we discuss relation of our work with concurrent one in the literature in Section~\ref{zou}. Numerical illustrations are given in Section~\ref{simulation}. The proofs of our results are postponed to Section~\ref{preuves}. A final discussion and perspectives of this work are provided in Section~\ref{group}. 

%%%%%%%%%%%%%%%%%%%%%%%%%%%%%%%%%%%%%%%%%%%%%%%%%%%%%%%%%%%%%%%%%%%%
\section{Main results}\label{resultats}
%%%%%%%%%%%%%%%%%%%%%%%%%%%%%%%%%%%%%%%%%%%%%%%%%%%%%%%%%%%%%%%%%%%%

%%%%%%%%%%%%%%%%%%%%%%%%%%%%%%%%%%
\subsection{An unbiased estimator of the $\dof$}\label{unb}
\added{First, some notations and definitions are necessary. For any vector $x$, $x_{i}$ denotes its $i$th component. The support or the active set of $x$ is defined by
\[
I = \supp ( x ) = \lbrace i : {x}_{i} \neq 0 \rbrace, 
\]
and we denote its cardinality as $\vert \supp( x ) \vert = \vert I \vert$. We denote by $x_{I} \in \R^{|I|}$ the vector built by restricting $x$ to the entries indexed by $I$. The active matrix $A_{I}=( a_{i} )_{i \in I}$ associated to a vector $x$ is obtained by selecting the columns of $A$ indexed by the support $I$ of $x$. Let ${\cdot}^{\Tr}$ be the transpose symbol. Suppose that $A_{I}$ is full column rank, then we denote the Moore-Penrose pseudo-inverse of $A_{I}$, $A_{I}^{+} = ( A_{I}^{\Tr} A_{I} )^{- 1} A_{I}^{\Tr}$. $\sign ( \cdot )$ represents the sign function: $\sign ( a ) = 1$ if $a > 0$; $\sign ( a ) = 0$ if $a = 0$; $\sign ( a ) = - 1$ if $a < 0$.\\}
For any $I \subseteq \lbrace 1, 2, \cdots, p \rbrace$, let $V_{I}=\span ( A_I )$, $P_{V_I}$ the orthogonal projector onto $V_{I}$ and $P_{V_{I}^\perp}$ that onto the orthogonal complement $V_{I}^{\perp}$.\\

Let $S \in \lbrace - 1, 1 \rbrace^{\vert I \vert}$ be a sign vector, and $j \in \lbrace 1, 2, \cdots, p \rbrace$. Fix $\lambda > 0$. We define the following set of hyperplanes
\begin{equation}\label{hyper}
H_{I,j,S} = \lbrace u \in \R^{n} : \langle P_{V_{I}^{\perp}} (a_j), u \rangle = \pm \lambda (1 - \langle a_j,   (A_{I}^{+})^{\Tr} S \rangle ) \rbrace.
\end{equation}
Note that, if $a_{j}$ does not belong to $V_{I}$, then $H_{I,j,S}$ becomes a finite union of two hyperplanes. Now, we define the following finite set of indices
\begin{equation}\label{hyper}
\Omega = \lbrace \left( I,j,S \right) : a_{j} \not \in V_{I} \rbrace
\end{equation}
and let $G_{\lambda}$ be the subset of $ \R^{n} $ which excludes the finite union of hyperplanes associate to $\Omega$, that is
\begin{equation}\label{g-lambda}
G_{\lambda} = \R^{n}\setminus \bigcup_{(I,j,S) \in \Omega} H_{I,j,S}.
\end{equation}
To cut a long story short, $\bigcup_{(I,j,S) \in \Omega} H_{I,j,S}$ is a set of (Lebesgue) measure zero (Hausdorff dimension $n-1$), and therefore $G_{\lambda}$ is a set of full measure.\\

We are now ready to introduce our main theorem.
\begin{theo}\label{theo-general}
Fix $\lambda > 0$. For any $y \in G_{\lambda}$, consider $\mathcal{M}_{y, \lambda}$ the set of  solutions of \eqref{eq1}. Let $x^{*}_{\lambda}(y) \in \mathcal{M}_{y, \lambda}$ with support $I^*$ such that $A_{I^{*}}$ is full rank. Then, 
\begin{equation}
\vert I^{*} \vert = \min_{\widehat{x}_{\lambda}(y)  \in \mathcal{M}_{y, \lambda}} \vert \supp ( \widehat{x}_{\lambda}(y)  ) \vert.
\end{equation}
Furthermore, there exists $\varepsilon > 0$ such that for all $z \in \ball (y, \varepsilon)$, the $n$-dimensional ball with center $y$ and radius $\varepsilon$, the Lasso response mapping $z \mapsto \widehat{\mu}_{\lambda} (z)$ satisfies
\begin{equation}
\widehat{\mu}_{\lambda} (z) = \widehat{\mu}_{\lambda}(y) + P_{V_{I^{*}}} (z - y).
\end{equation}
\end{theo}
\added{As stated, this theorem assumes the existence of a solution whose active matrix $A_{I^*}$ is full rank. This can be shown to be true; see e.g. \cite[Proof of Theorem 1]{dossal} or \cite[Theorem 3, Section B.1]{rosset04}\footnote{This proof is alluded to in the note at the top of \cite[Page 363]{sardy}.}. It is worth noting that this proof is constructive, in that it yields a solution $x^*_\lambda(y)$ of \eqref{eq1} such that $A_{I^*}$ is full column rank from any solution $\widehat{x}_\lambda(y)$ whose active matrix has a nontrivial kernel. This will be exploited in Section~\ref{simulation} to derive an algorithm to get $x^*_\lambda(y)$, and hence $I^*$.}\\

A direct consequence of our main theorem is that outside $G_\lambda$, the mapping $\widehat{\mu}_{\lambda}(y)$ is $C^\infty$ and the sign and support are locally constant. Applying Stein's lemma yields Corollary~\ref{Corollary1} below. The latter states that the number of nonzero coefficients of $x^{*}_{\lambda}(y)$ is an unbiased estimator of the $\dof$ of the Lasso. 
\begin{cor}\label{Corollary1}
%Fix $\lambda > 0$. For any $y$ belongs to the full measure set $G_{\lambda}$, 
Under the assumptions and with the same notations as in Theorem~\ref{theo-general}, we have the following divergence formula
\begin{equation}
\widehat{df}_\lambda(y) := \divg ( \widehat{\mu}_{\lambda}(y) ) = \vert I^{*} \vert.
\end{equation}
Therefore,
\begin{equation}\label{dfff}
df = \E(\widehat{df}_\lambda(y)) = \E ( \vert I^{*} \vert ).
\end{equation}
\end{cor}
Obviously, in the particular case where the Lasso problem has a unique solution, our result holds true. 
%Precisely, the cardinality of the support of this solution is an unbiased estimator of the degrees of freedom of the Lasso response.
%Obviously, in the particular case where the Lasso problem has a unique solution, our result remains true.
%\begin{cor}\label{Corollary2}
%Fix $\lambda > 0$. For any $y \in G_{\lambda}$, suppose that the Lasso has a unique solution $\widehat{x}_{\lambda}$ with support $\widehat{I}$. Then,
%\begin{equation*}
%df ( \widehat{\mu}_{\lambda} ) = \E ( \vert \widehat{I} \vert ).
%\end{equation*}
%\end{cor}

%%%%%%%%%%%%%%%%%%%%%%%%%%%%%%%%%%
\subsection{Reliability of the $\sure$ estimate of the Lasso prediction risk}\label{fiabilite}
In this work, we focus on the $\sure$ as a model selection criterion. The $\sure$ applied to the Lasso reads
\begin{equation}\label{sure}
\sure (\widehat{\mu}_\lambda(y)) = - n \sigma^{2} + \Vert \widehat{\mu}_\lambda(y) - y \Vert_{2}^{2} + 2 \sigma^{2} \widehat{df}_\lambda(y) ,
\end{equation}
where $\widehat{df}(y)$ is an unbiased estimator of the $\dof$ as given in Corollary~\ref{Corollary1}. It follows that the $\sure(\widehat{\mu}_\lambda(y))$ is an unbiased estimate of the prediction risk, i.e.
\[
\mse (\mu) = \E \parenth{\Vert \widehat{\mu}_\lambda(y) - \mu \Vert_{2}^{2}} = \E\parenth{\sure (\widehat{\mu}_\lambda(y))}.
\]
%In Corollary~\ref{Corollary1}, we gave an unbiased estimator of the degree of freedom $\widehat{df}_\lambda (y)$ of the Lasso. An immediate consequence is that the $\sure ( \widehat{\mu}_{\lambda}(y) )$ is an unbiased estimator of the prediction risk $\mse (\mu)$. 
We now evaluate its reliability by computing the expected squared-error between $\sure(\widehat{\mu}_\lambda(y))$ and $\se(\widehat{\mu}_\lambda(y))$, the true squared-error, that is
\begin{equation}\label{se}
\se(\widehat{\mu}_\lambda(y)) = \Vert \widehat{\mu}_{\lambda}(y) - \mu \Vert_{2}^{2}.
\end{equation}
%From the estimator of the degree of freedom of the Lasso response $\widehat{df} ( \widehat{\mu}_{\lambda} )$, it follows that the $\sure ( \widehat{\mu}_{\lambda} )$ (see \eqref{sure}) is an unbiased estimator of $\mse (\widehat{\mu}_{\lambda}(y))$ the true risk, defined by \eqref{ri}. We now evaluate its reliability by computing the expected squared-error between $\sure$ and $\se$, the true squared-error, that is
%\begin{equation}\label{se}
%\se = \Vert \widehat{\mu}_{\lambda} - \mu \Vert_{2}^{2}.
%\end{equation}

\begin{theo}\label{th2}
Under the assumptions of Theorem~\ref{theo-general}, we have
\begin{equation}\label{rh}
\E \left( \left( \sure ( \widehat{\mu}_{\lambda}(y) ) - \se(\widehat{\mu}_\lambda(y)) \right)^{2} \right) = - 2 \sigma^{4} n + 4 \sigma^{2} \E \left( \Vert \widehat{\mu}_{\lambda}(y)- y \Vert_{2}^{2} \right) + 4 \sigma^{4} \E \left( \vert I^{*} \vert \right).
\end{equation}
Moreover,
\begin{equation}
\E \left( \left( \dfrac{\sure ( \widehat{\mu}_{\lambda}(y)) - \se(\widehat{\mu}_\lambda(y))}{n \sigma^{2}} \right)^{2} \right) = \bigO{\frac{1}{n}}.
\end{equation}
\end{theo}

%%%%%%%%%%%%%%%%%%%%%%%%%%%%%%%%%%%%%%%%%%%%%%%%%%%%%%%%%%%%%%%%%%%%
\section{Relation to prior work}\label{zou}
%%%%%%%%%%%%%%%%%%%%%%%%%%%%%%%%%%%%%%%%%%%%%%%%%%%%%%%%%%%%%%%%%%%%

%%%%%%%%%%%%%%%%%%%%%%%%%%%%%%%%%%
\subsection*{Overdetermined case \cite{zou}}
The authors in \cite{zou} studied the $\dof$ of the Lasso in the overdetermined case. Precisely, when $n \geq p$ and all the columns of the design matrix $A$ are linearly independent, i.e. $\rank (A) = p$. In fact, in this case the Lasso problem has a unique minimizer $\widehat{x}_{\lambda}(y)=x^*_{\lambda}(y)$ (see Theorem~\ref{theo-general}).

Before discussing the result of \cite{zou}, let's point out a popular feature of $\widehat{x}_{\lambda}(y)$ as $ \lambda $ varies in $]0,+\infty[$:
\begin{itemize}
\item For $\lambda \geq \Vert A^{\Tr} y  \Vert_{\infty}$, the optimum is attained at $\widehat{x}_\lambda(y) = 0$.
\item The interval $\left] 0, \Vert A^{\Tr} y \Vert_{\infty} \right[$ is divided into a finite number of subintervals characterized by the fact that within each such subinterval, the support and the sign vector of $\widehat{x}_{\lambda}(y)$ are constant. Explicitly, let $\parenth{\lambda_{m}}_{0 \leq m \leq K}$ be the finite sequence of $\lambda$'s values corresponding to a variation of the support and the sign of $\widehat{x}_\lambda(y)$, defined by
\begin{equation*}
\Vert A^{\Tr} y  \Vert_{\infty} = \lambda_{0} > \lambda_{1} > \lambda_{2} > \cdots > \lambda_{K} = 0.
\end{equation*}
Thus, in $] \lambda_{m+1}, \lambda_{m} [$, the support and the sign of $\widehat{x}_{\lambda}(y)$ are constant, see \cite{efronn,osborne1,osborne2}. Hence, we call $\parenth{\lambda_{m}}_{0 \leq m \leq K}$ the \textit{transition points}.
\end{itemize}
Now, let $\lambda \in ] \lambda_{m+1}, \lambda_{m} [$. Thus, from Lemma~\ref{lem1} \added{(see Section~\ref{preuves})}, we have the following implicit form of $\widehat{x}_\lambda(y)$,
\begin{equation}\label{eq16}
( \widehat{x}_\lambda(y))_{I_{m}}  = A_{I_{m}}^{+} y - \lambda ( A_{I_{m}}^{\Tr} A_{I_{m}} )^{- 1} S{m},
\end{equation}
where $I_{m}$ and $S^{m}$ are respectively the (constant) support and sign vector of $\widehat{x}_\lambda(y)$ for $\lambda \in ] \lambda_{m+1}, \lambda_{m} [$. Hence, based on \eqref{eq16}, \cite{zou} showed that for all $\lambda > 0$, there exists a set of measure zero $\mathcal{N}_{\lambda}$, which is a finite collection of hyperplanes in $\R^{n}$, and they defined 
\begin{equation}\label{kappa}
\mathcal{K}_{\lambda} = \R^{n} \setminus \mathcal{N}_{\lambda},
\end{equation}
so that $\forall~ y \in \mathcal{K}_{\lambda}$, $\lambda$ is not any of the transition points.\\
Then, for the overdetermined case, \cite{zou} stated that for all $y \in \mathcal{K}_{\lambda}$,  the number of nonzero coefficients of the unique solution of \eqref{eq1} is an unbiased estimator of the $\dof$. In fact, their main argument is that, by eliminating the vectors associated to the transition points, the support and the sign of the Lasso solution are locally constant with respect to $y$, see \cite[Lemma 5]{zou}.

We recall that the overdetermined case, considered in \cite{zou}, is a particular case of our result since the minimizer is unique. Thus, according to the Corollary~\ref{Corollary1}, we find the same result as \cite{zou} but valid on a different set $y \in G_{\lambda} = \R^{n}\setminus \bigcup_{(I,j,S) \in \Omega} H_{I,j,S}$. A natural question arises: can we compare our assumption to that of \cite{zou} ? In other words, is there a link between $\mathcal{K}_{\lambda}$ and $G_{\lambda}$ ?\\

%First, we should mention that for the vast majority of 
%design matrices, $G_{\lambda}$ excludes only the vectors 
%associated to the transition points, and then both 
%$\mathcal{K}_{\lambda}$ and $G_{\lambda}$ coincide. 
%On the other hand, we believe that in general way, 
%consider $\mathcal{K}_{\lambda}$ is not sufficient 
%to ensure does the stability of the support and the 
%sign of the Lasso solution, with respect to an infinitesimal 
%perturbation of y. Namely, we can construct a design matrix, 
%such that there exists $y \in \mathcal{K}_{\lambda} \cap 
%\left( G_{\lambda} \right)^{c}$ and $\varepsilon$ small 
%enough, so that for all $z \in \ball (y,\varepsilon)$ 
%the support of the unique solution of the Lasso problem 
%$(\mbox{P}_{1} ( z, \lambda ))$ is different from that 
%of the unique solution of $(\mbox{P}_{1} ( z, \lambda ))$.
The answer is that, depending on the matrix $A$, these two sets may be different. \added{More importantly, it turns out that although the $\dof$ formula \cite[Theorem 1]{zou} is correct, unfortunately, their proof contains a flaw since their divergence formula \cite[Lemma 5]{zou} is not true on the set $\mathcal{K}_{\lambda}$. We prove this by providing a simple counterexample.}

%%%%%%%%%%%%%%%%%%%%%%%%%%%%%%%%%%
\paragraph{Example of vectors in $G_{\lambda}$ but not in $\mathcal{K}_{\lambda}$}
Let $\{e_1,e_2\}$ be an orthonormal basis of $\mathbb{R}^2$ and 
let's define $a_1=e_1$ and $a_2=e_1+e_2$, and $A$ the matrix 
whose columns are $a_1$ and $a_2$.

Let's define $I=\{1\}$, $j=2$ and $S=1$. It turns out that $A_I^+=a_1$
and $\langle (A_I^+)^\Tr S,a_j\rangle=1$ which implies that for all $\lambda>0$,
\begin{equation*}
H_{I,j,S}=\{u\in\mathbb{R}^n\,: \langle P_{V_I^{\perp}}(a_j),u\rangle=0\} = \span (a_1) ~.
\end{equation*}
Let $y=\alpha a_1$ with $\alpha>0$, for any $\lambda>0$, $y\in H_{I,j,S}$ (or equivalently here $y \notin G_{\lambda}$). Using Lemma~\ref{lem1} \added{(see Section~\ref{preuves})}, one gets that for any $\lambda\in]0,\alpha[$, 
the solution of \eqref{eq1} is $\widehat{x}_\lambda(y)=(\alpha-\lambda,0)$ and that for any $\lambda\geq \alpha$, $\widehat{x}_\lambda(y)=(0,0)$. Hence the only transition point is $\lambda_0=\alpha$. It follows that for $\lambda<\alpha$, $y$ belongs to $\mathcal{K}_{\lambda}$ defined in \cite{zou}, but $y\notin G_{\lambda}$.

We prove then that in any ball centered at $y$, there exists a vector 
$z_{1}$ such that the support of the solution of $(\mathrm{P}_1(z_{1},\lambda))$ is 
different from the support of \eqref{eq1}.\\
Let's choose $\lambda<\alpha$ and $\varepsilon\in]0,\alpha-\lambda[$ and 
let's define $z_1=y+\varepsilon e_2$. From Lemma~\ref{lem1} \added{(see Section~\ref{preuves})}, 
one deduces that the solution of $(\mathrm{P}_1(z_{1},\lambda))$ is equal to 
$\widehat{x}_\lambda(z_1)=(\alpha-\lambda-\varepsilon,\varepsilon)$ whose support 
is different from that of $\widehat{x}_\lambda(y)=(\alpha-\lambda,0)$.\\

More generally, when there are sets $\{I,j,S\}$ such that $\langle (A_I^+)^\Tr S,a_j\rangle=1$, a difference between the two sets $G_{\lambda}$ and $\mathcal{K}_{\lambda}$ may arise. Clearly, $G_{\lambda}$ is not only the set of transition points associated to $\lambda$. 

According to the previous example, in this specific situation, for any $\lambda>0$ there may exist some vectors $y$ that are not transition points associated to $\lambda$ where the support of the solution of \eqref{eq1} is not stable to infinitesimal perturbations of $y$. This situation may occur for under or overdetermined problems. In summary, even in the overdetermined case, excluding the set of transition points is not sufficient to guarantee stability of the support and sign of the Lasso solution. 

\begin{figure}[htbp]
\begin{center}
\includegraphics[width=0.2\textwidth]{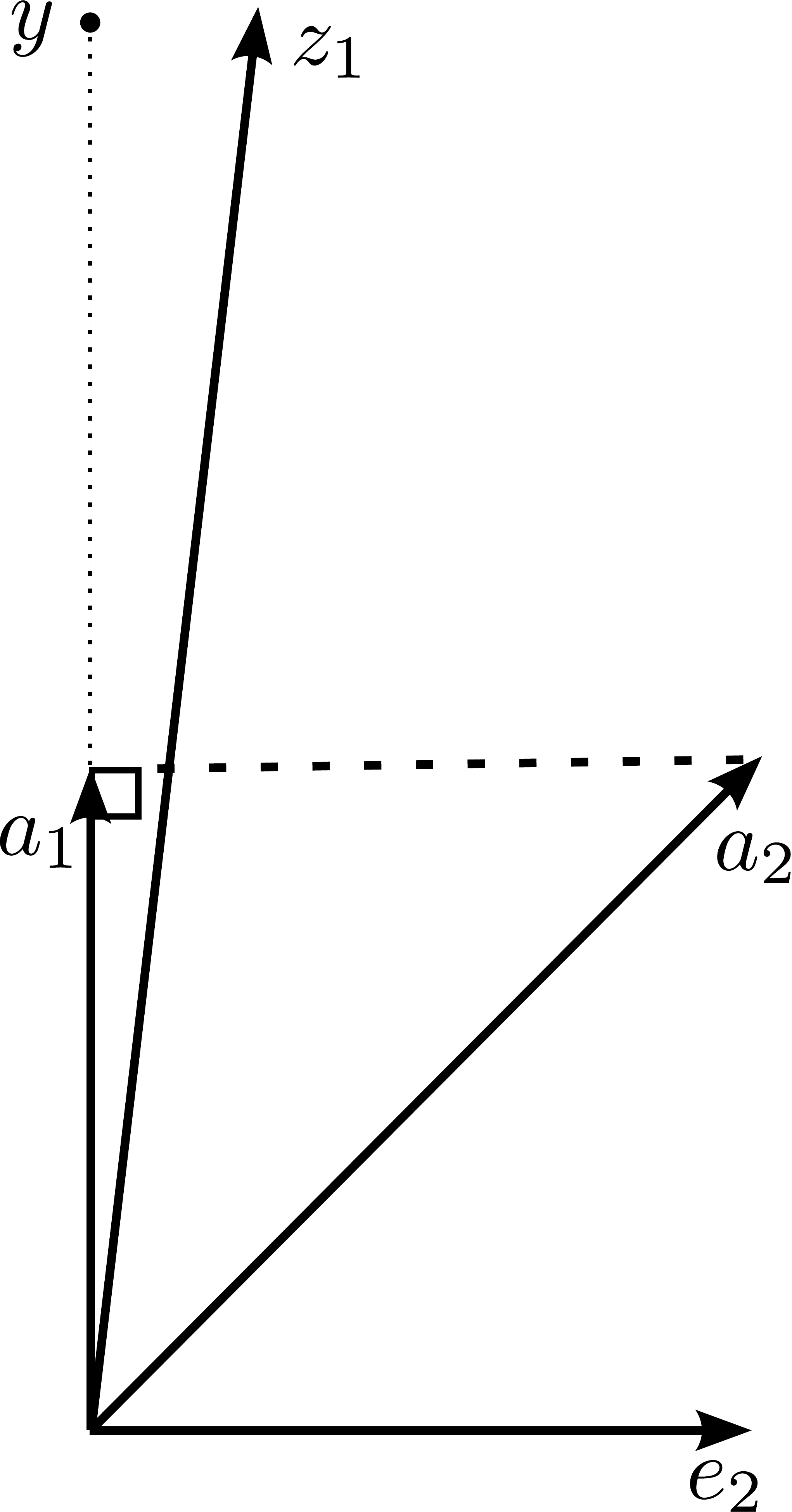}
\caption{A counterexample for $n=p=2$ of vectors in $G_{\lambda}$ but not in $\mathcal{K}_{\lambda}$. See text for a detailed discussion.}
\end{center}
\label{fig2}
\end{figure}
%\\
%{\bf Charles : Jalal, ici faire une remarque sur le fait qu'on ne 
%peut pas definir le degre de liberte en $y$. 
%je te laisse voir comment on est désagréable ou pas avec zou. il faudra sans 
%doute reprendre le paragraphe precedent. je ne sais pas si ca a un 
%interet d'appuyer la dessus mais je pense que dans le cas ou (uc) 
%est vérifiée, c'est a dire qu'il n'existe pas de tels vecteurs, 
%$G_\lambda$ est exactement egal aux points de transition.}  

%%%%%%%%%%%%%%%%%%%%%%%%%%%%%%%%%%
\added{\subsection*{General case \cite{kato09,tib12,vaiter11}}
In \cite{kato09}, the author studies the degrees of freedom of a generalization of the Lasso where the regression
coefficients are constrained to a closed convex set. When the latter is a $\ell_1$ ball and $p > n$, he proposes the cardinality of the support as an estimate of $df$ but under a restrictive assumption on $A$ under which the Lasso problem has a unique solution.

In \cite[Theorem 2]{tib12}, the authors proved that
\[
df = \E (\rank(A_I))
\]
where $I=I(y)$ is the active set of any solution $\widehat{x}_\lambda(y)$ to \eqref{eq1}. This coincides with Corollary~\ref{Corollary1} when $A_I$ is full rank with $\rank(A_I)=\rank(A_{I^*})$. Note that in general, there exist vectors $y \in \mathbb{R}^n$ where the smallest cardinality among all supports of Lasso solutions is different from the rank of the active matrix associated to the largest support. But these vectors are precisely those excluded in $G_{\lambda}$. In the case of the generalized Lasso (a.k.a. analysis sparsity prior in the signal processing community), Vaiter et al. \cite[Corollary 1]{vaiter11} and Tibshirani and Taylor \cite[Theorem 3]{tib12} provide a formula of an unbiased estimator of $df$. This formula reduces to that of Corollary~\ref{Corollary1} when the analysis operator is the identity.} 

\begin{figure}[htbp]
\subfigure[Gaussian]{\includegraphics[width=\textwidth]{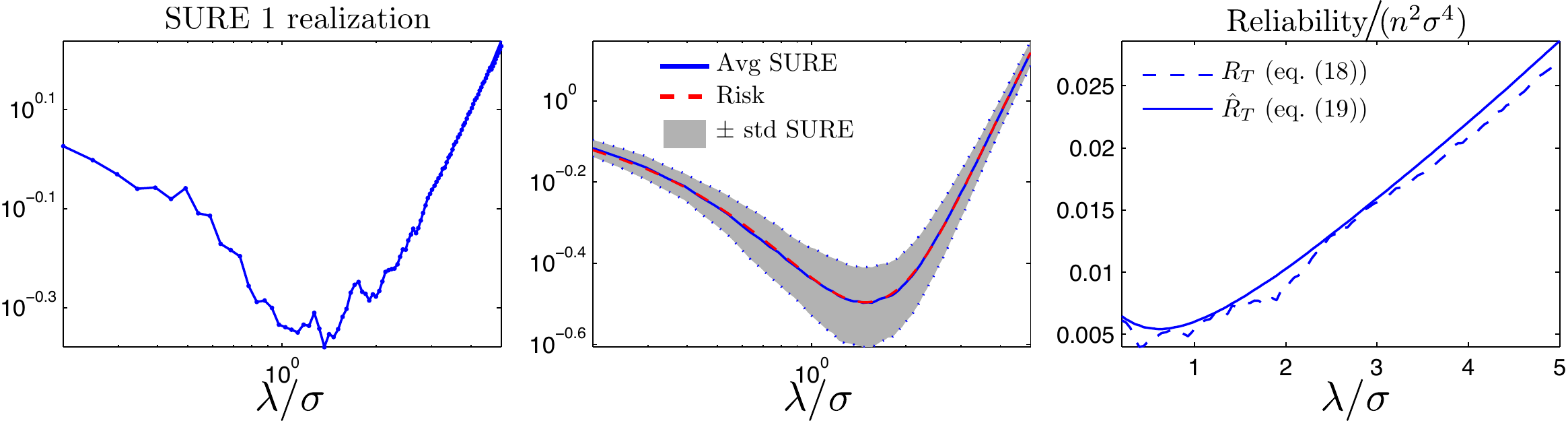}}
%\subfigure[Random Fourier]{\includegraphics[width=0.6\textwidth]{dofSURERST.pdf}}
%\hspace*{-2cm}
%\subfigure[Random Hadamard]{\includegraphics[width=0.6\textwidth]{dofSUREHadamard.pdf}}
\subfigure[Convolution]{\includegraphics[width=\textwidth]{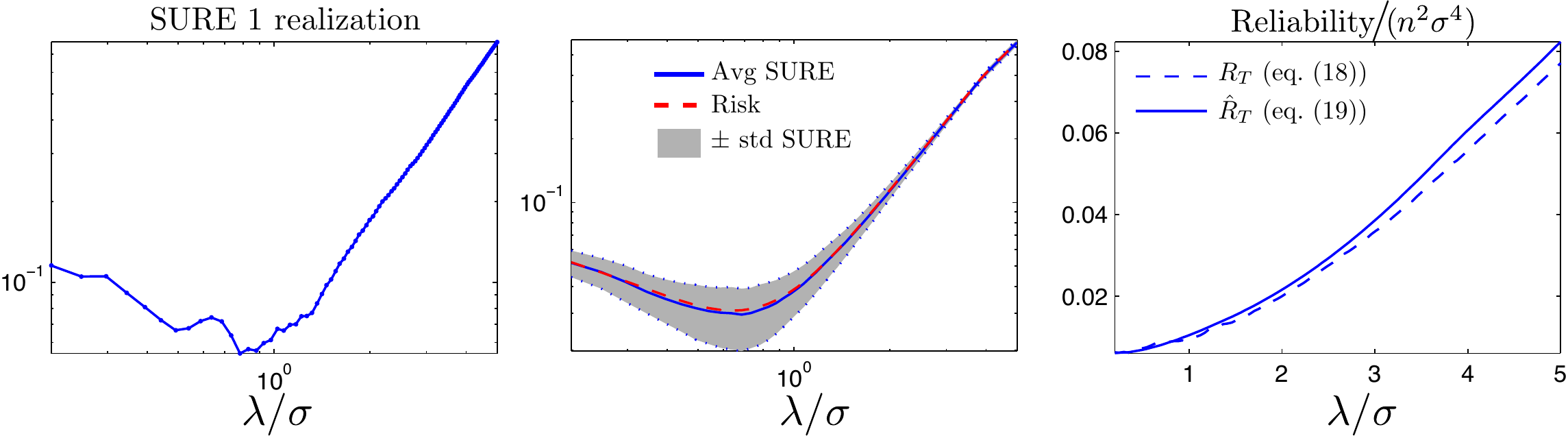}}
\caption{\added{The $\sure$ and its reliability as a function of $\lambda$ for two types of design matrices. (a) Gaussian; (b) Convolution. For each kind of design matrix, we associate three plots.}}
\label{fig1}
\end{figure}

%Note that here, since $n$ varies, the curve of the upper bound of the normalized reliability, given by right hand term of \eqref{bs}, is no longer a horizontal line.
\begin{figure}[htbp]
\begin{centering}
\subfigure[$\lambda/\sigma=0.1$]{\includegraphics[width=0.75\textwidth]{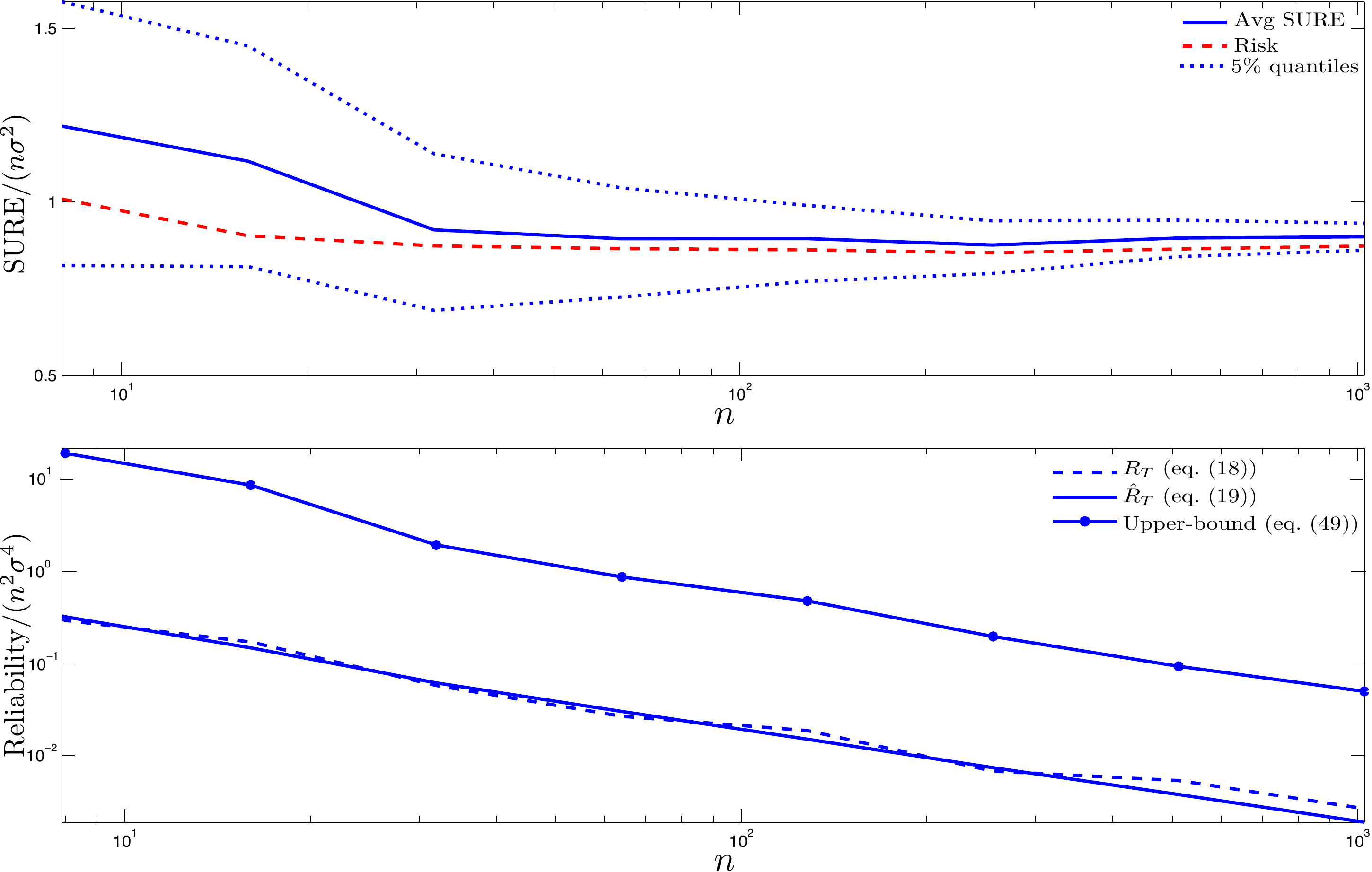}}\\
\subfigure[$\lambda/\sigma=1$]{\includegraphics[width=0.75\textwidth]{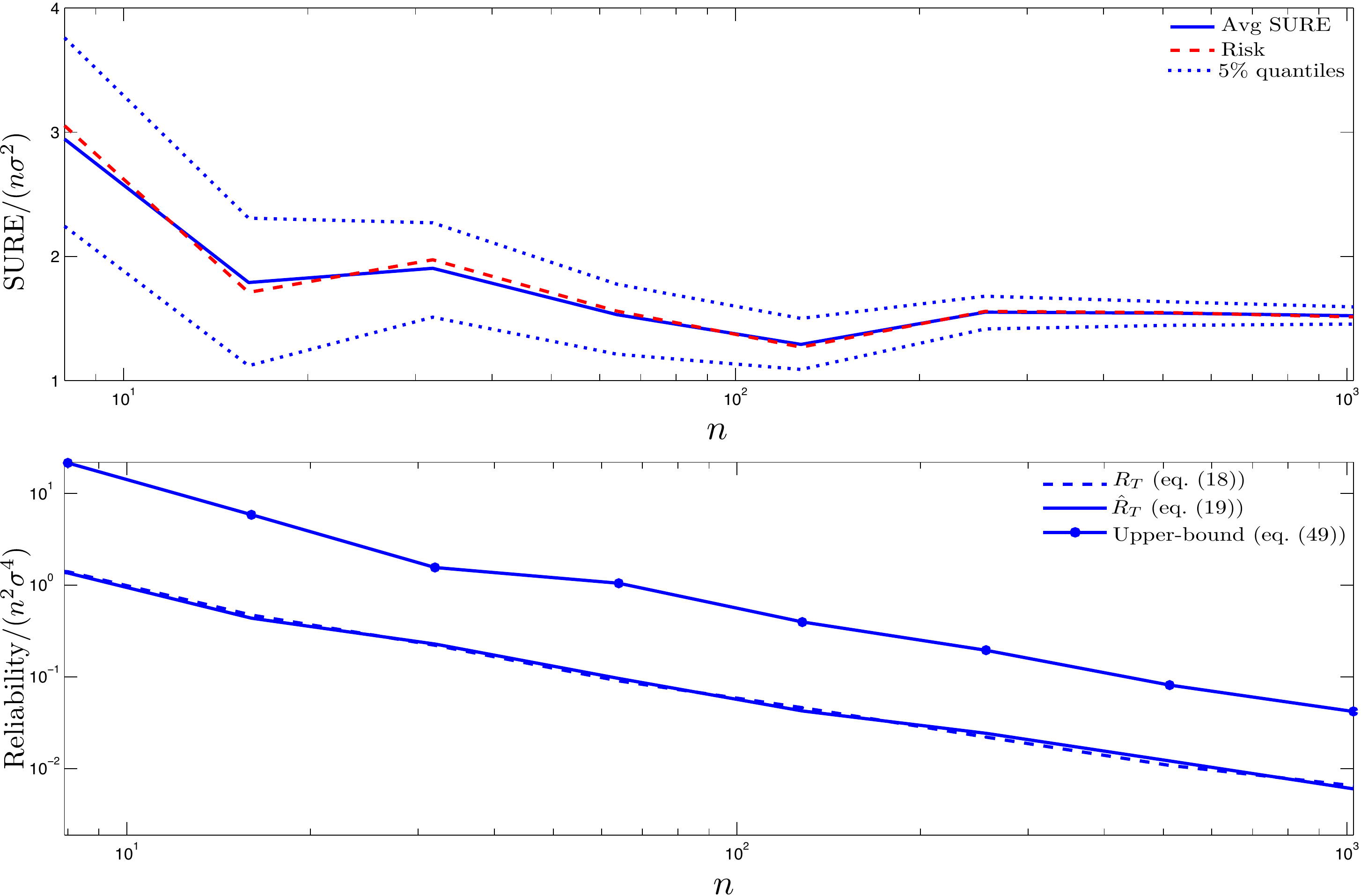}}\\
\subfigure[$\lambda/\sigma=10$]{\includegraphics[width=0.75\textwidth]{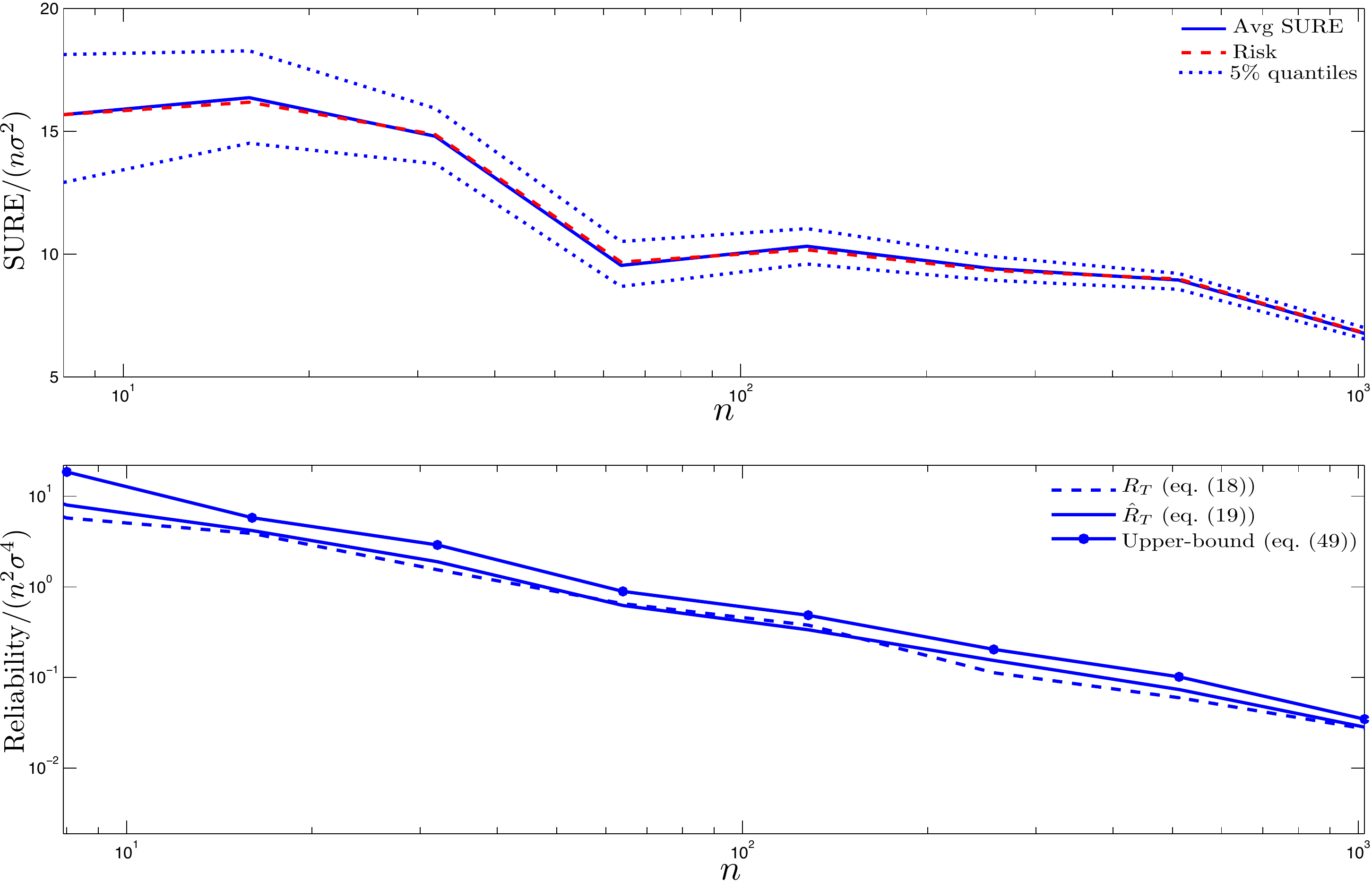}} \\
\end{centering}
\caption{The $\sure$ and its reliability as a function of the number of observations $n$.}
\label{fig3}
\end{figure}

%%%%%%%%%%%%%%%%%%%%%%%%%%%%%%%%%%%%%%%%%%%%%%%%%%%%%%%%%%%%%%
\section{Numerical experiments}\label{simulation}
%%%%%%%%%%%%%%%%%%%%%%%%%%%%%%%%%%%%%%%%%%%%%%%%%%%%%%%%%%%%%%

%%%%%%%%%%%%%%%%%%%%%%%%%%%%%%%
\paragraph{Experiments description} 
\added{In this section, we support the validity of our main theoretical findings with some numerical simulations, by checking the unbiasedness and the reliability of the $\sure$ for the Lasso. Here is the outline of these experiments.}

For our first study, \added{we consider two kinds of design matrices $A$, a random Gaussian matrix with $n = 256$ and $p = 1024$ whose entries are $\sim_{\mathrm{iid}}\mathcal{N}(0,1/n)$, and a deterministic convolution design matrix $A$ with $n = p = 256$ and a Gaussian blurring function.} The original sparse vector $x^0$ was drawn randomly according to a mixed Gaussian-Bernoulli distribution, such that $x^0$ is $15$-sparse (i.e. $|\supp(x^0)=15|$). 
%Hence, through this study we fix the denoised data $\mu = A x^0$, for each case. Then, to an error variable, say $\epsilon$, we consider the standard gaussian random distribution. 
For each design matrix $A$ and vector $x^0$, we generate $K = 100$ independent replications $y^k \in \R^n$ of the observation vector according to the linear regression model \eqref{eq0}. Then, for each $y^k$ and a given $\lambda$, we compute the Lasso response $\widehat{\mu}_{\lambda}(y^k)$ using the now popular iterative soft-thresholding algorithm \added{\cite{daubechies}}\footnote{Iterative soft-thresholding through block-coordinate relaxation was proposed in \cite{sardy} for matrices $A$ structured as the union of a finite number of orthonormal matrices.}, and we compute $\sure(\widehat{\mu}_{\lambda}(y^k))$ and $\se(\widehat{\mu}_{\lambda}(y^k))$. We then compute the empirical mean and the standard deviation of $\parenth{\sure(\widehat{\mu}_{\lambda}(y^k))}_{1 \leq k \leq K}$, the empirical mean of $\parenth{\se(\widehat{\mu}_{\lambda}(y^k))}_{1 \leq k \leq K}$, which corresponds to the computed prediction risk, and we compute $R_{T}$ the empirical normalized reliability on the left-hand side of \eqref{rh},
\begin{equation}\label{rt}
R_{T} = \dfrac{1}{K} \sum_{k = 1}^{K} \left( \dfrac{\sure(\widehat{\mu}_{\lambda}(y^k)) - \se(\widehat{\mu}_{\lambda}(y^k))}{n \sigma^{2}} \right)^{2}.
\end{equation}
Moreover, based on the right-hand side of \eqref{rh}, we compute $\widehat{R}_{T}$ as
\begin{equation}\label{ert}
\widehat{R}_{T} = - \dfrac{2}{n} + \dfrac{4}{n^{2} \sigma^{2}} \left( \dfrac{1}{K} \sum_{k = 1}^{K} \left( \Vert \widehat{\mu}_{\lambda}(y^k) - y^{k} \Vert_{2}^{2} \right) \right) + \dfrac{4}{n^2} \left( \dfrac{1}{K} \sum_{k = 1}^{K} \left( \vert I^{*} \vert_{k} \right) \right),
\end{equation}
where at the $k$th replication, $\vert I^{*} \vert_{k}$ is the cardinality of the support of a Lasso solution whose active matrix is full column rank as stated in Theorem~\ref{theo-general}. Finally, we repeat all these computations for various values of $\lambda$, for the \added{two} kinds of design matrices considered above.\\

%%%%%%%%%%%%%%%%%%%%%%%%%%%%%%%
\paragraph{Construction of full rank active matrix} 
As stated in the discussion just after Theorem~\ref{theo-general}, in situations where the Lasso problem has non-unique solutions, and the minimization algorithm returns a solution whose active matrix is rank deficient, one can construct an alternative optimal solution whose active matrix is full column rank, and then get the estimator of the degrees of freedom. 

More precisely, let $\widehat{x}_\lambda(y)$ be a solution of the Lasso problem with support $I$ such that its active matrix $A_I$ has a non-trivial kernel. The construction is as follows:
\begin{enumerate}
\item Take $h \in \ker{A_I}$ such that $\supp{h} \subset I$. 
\item For $t \in \R$, $A \widehat{x}_\lambda(y)= A\parenth{\widehat{x}_\lambda(y)+th}$ and the mapping $t \mapsto \|\widehat{x}_\lambda(y)+th\|_1$ is locally affine in a neighborhood of 0, i.e. for $|t| < \min_{j \in I} |(\widehat{x}_\lambda(y))_j|/\|h\|_\infty$. $\widehat{x}_\lambda(y)$ being a minimizer of \eqref{eq1}, this mapping is constant in a neighborhood of 0. We have then constructed a whole collection of solutions to \eqref{eq1} having the same image and the same $\ell_1$ norm, which lives on a segment. 
\item Move along $h$ with the largest step $t_0 > 0$ until an entry of $\widehat{x}^1_\lambda(y) = \widehat{x}_\lambda(y) + t_0 h$ vanishes, i.e. $\supp(\widehat{x}^1_\lambda(y) + t_0 h) \subsetneq I$. 
\item Repeat this process until getting a vector $x^*_\lambda(y)$ with a full column rank active matrix $A_{I^*}$. 
\end{enumerate}
Note that this construction bears similarities with the one in \cite{rosset04}.\\

%%%%%%%%%%%%%%%%%%%%%%%%%%%%%%%
\paragraph{Results discussion} 
\added{Figure~\ref{fig1} depicts the obtained results. For each design matrix, we associate a panel, each containing three plots. Hence, for each case, from left to right, the first plot represents the $\sure$ for one realization of the noise as a function of $\lambda$. In the second graph, we plot the computed prediction risk curve and the empirical mean of the $\sure$ as a function of the regularization parameter $\lambda$. Namely, the dashed curve represents the calculated prediction risk, the solid curve represents the empirical mean of the $\sure$, and the shaded area represent the empirical mean of the sure $\pm$ the empirical standard deviation of the $\sure$. The latter shows that the $\sure$ is an unbiased estimator of the prediction risk with a controlled variance. This suggests that the $\sure$ is consistent, and then so is our estimator of the degrees of freedom. In the third graph, we plot the theoretical and empirical normalized reliability, defined respectively by \eqref{rt} and \eqref{ert}, as a function of the regularization parameter $\lambda$. More precisely, the solid and dashed blue curves represent respectively $R_{T}$ and $\widehat{R}_{T}$. This confirms numerically that both sides ($R_{T}$ and $\widehat{R}_{T}$) of \eqref{rh} indeed coincide.} 

As discussed in the introduction, one of the motivations of having an unbiased estimator of the degrees of freedom of the Lasso is to provide a data-driven objective way for selecting the optimal Lasso regularization parameter $\lambda$. For this, one can compute the optimal $\lambda$ that minimizes the $\sure$, i.e.
\begin{equation}
\lambda_{{\mbox{\tiny optimal}}} = \argmin_{\lambda > 0} ~ \sure (\widehat{\mu}_{\lambda}(y)).
\end{equation}
\added{In practice, this optimal value can be found either by a exhaustive search over a fine grid, or alternatively by any dicothomic search algorithm (e.g. golden section) if $\lambda \mapsto \sure (\widehat{\mu}_{\lambda}(y))$ is unimodal.}\\

%\deleted{In the fourth graph, we compare the original signal $x^0$, represented by the blue circles, and a Lasso solution associated to $\lambda_{{\mbox{\tiny optimal}}}$, denoted by $\widehat{x}_{\lambda_{{\mbox{\tiny optimal}}}}$ plotted with red crosses. We remark that some coefficients of $\widehat{x}_{\lambda_{{\mbox{\tiny optimal}}}}$ are nonzero outside the support of $x^0$. This is not a real surprise, since the optimality is in the sense of the prediction variable estimation rather than the regression coefficients.}
%Hence, we can deduce that $\widehat{x}_{\lambda_{{\mbox{\tiny optimal}}}}$ does not recover $x^0$, that is, both vectors does not have the same support and sign. 

\added{Now, for our second simulation study, we consider a partial Fourier design matrix, with $n < p$ and a constant underdeterminacy factor $p/n=4$. $x^0$ was again simulated according to a mixed Gaussian-Bernoulli distribution with $\lceil 0.1 p\rceil$ non-zero entries. For each of three values of $\lambda/\sigma \in \{0.1,1,10\}$ (small, medium and large), we compute the prediction risk curve, the empirical mean of the $\sure$, as well as the values of the normalized reliability $R_{T}$ and $\widehat{R}_{T}$, as a function of $n \in \{8,\cdots,1024\}$. The obtained results are shown in Figure~\ref{fig3}. For each value of $\lambda$, the first plot (top panel) displays the normalized empirical mean of the $\sure$ (solid line) and its $5\%$ quantiles (dotted) as well as the computed normalized prediction risk (dashed). Unbiasedness is again clear whatever the value of $\lambda$. The trend on the prediction risk (and average $\sure$) is in agreement with rates known for the Lasso, see e.g. \cite{bickel09}. The second plot confirms that the $\sure$ is an asymptotically reliable estimate of the prediction risk with the rate established in Theorem~\ref{th2}. Moreover, as expected, the actual reliability gets closer to the upper-bound \eqref{bs} as the number of samples $n$ increases.}

%%%%%%%%%%%%%%%%%%%%%%%%%%%%%%%%%%%%%%%%%%%%%%%%%%%%%%%%%%%%%%%%%%%%%%%%%%%%%%%%%%%%%%%%%%%%%%%%
%%%%%%%%%%%%%%%%%%%%%%%%%%%%%%%%%%%%%%%%%%%%%%%%%%%%%%%%%%%%%%%%%%%%%%%%%%%%%%%%%%%%%%%%%%%%%%%%
\section{Proofs}\label{preuves}
%%%%%%%%%%%%%%%%%%%%%%%%%%%%%%%%%%%%%%%%%%%%%%%%%%%%%%%%%%%%%%%%%%%%%%%%%%%%%%%%%%%%%%%%%%%%%%%%%%%%%%%%%
%%%%%%%%%%%%%%%%%%%%%%%%%%%%%%%%%%%%%%%%%%%%%%%%%%%%%%%%%%%%%%%%%%%%%%%%%%%%%%%%%%%%%%%%%%%%%%%%%%%%%%%%%
\added{First of all, we recall some classical properties of any solution of the Lasso (see, e.g., \cite{osborne1,efronn,fuchs,tropp}). To lighten the notation in the two following lemmas, we will drop the dependency of the minimizers of \eqref{eq1} on either $\lambda$ or $y$.
%Next, we evoke the first order optimality conditions for the Lasso estimator, see \cite{fletcher} and  \cite{fuchs}. 
\begin{lem}\label{lem1} $\widehat{x}$ is a (global) minimizer of the Lasso problem \eqref{eq1} if and only of:
\begin{enumerate}
\item $A_{I}^{\Tr} ( y - A \widehat{x} ) = \lambda \sign ( \widehat{x}_{I} ) $, where $I = \lbrace i : \widehat{x}_{i} \neq 0 \rbrace$, and\label{condition1-1}
\item $\vert \langle a_{j}, y - A \widehat{x} \rangle \vert \leq \lambda$, $\forall~ j \in I^{c}$,\label{condition1-2}
\end{enumerate}
where $I^{c} = \{1,\ldots,p\} \setminus I$. Moreover, if $A_{I}$ is full column rank, then $\widehat{x}$ satisfies the following implicit relationship:
\begin{equation}\label{eq11}
\widehat{x}_{I} = A_{I}^{+} y - \lambda ( A_{I}^{\Tr} A_{I} )^{- 1} \sign ( \widehat{x}_{I} ) ~.
\end{equation}
%where $A_{I}^{+}$ is the pseudo-inverse of $A_{I}$.
\end{lem}
Note that if the inequality in condition~\ref{condition1-2} above is strict, then $\widehat{x}$ is the unique minimizer of the Lasso problem \eqref{eq1} \cite{fuchs}. \\

Lemma~\ref{lem3} below shows that all solutions of \eqref{eq1} have the same image by $A$. In other words, the Lasso response $\widehat{\mu}_{\lambda}(y)$, is unique, see \cite{dossal}.
\begin{lem}\label{lem3}
If $\widehat{x}^{1}$ and $\widehat{x}^{2}$ are two solutions of \eqref{eq1}, then 
\[
A \widehat{x}^{1} = A \widehat{x}^{2} = \widehat{\mu}_{\lambda}(y).
\]
\end{lem}}

Before delving into the technical details, we recall the following trace formula of the divergence. Let $J_{\widehat{\mu}(y)}$ be the Jacobian matrix of a mapping $y \mapsto \widehat{\mu}(y)$, defined as follows
\begin{equation}\label{derive}
\left( J_{\widehat{\mu}(y)} \right)_{i,j} := \dfrac{\partial \widehat{\mu}(y)_{i}}{\partial y_{j}}, \quad \quad i, j = 1, \cdots, n.
\end{equation}
Then we can write
\begin{equation}\label{divv}
\divg \left( \widehat{\mu}(y) \right) = \tr \left( J_{\widehat{\mu}(y)} \right) .
\end{equation}
%The above trace expression will be used in our proofs.

\begin{proof}[Proof of Theorem~\ref{theo-general}]
Let $x^{*}_{\lambda}(y)$ be a solution of the Lasso problem \eqref{eq1} and $I^{*}$ its support such that $A_{I^*}$ is full column rank. Let $( x^{*}_{\lambda} (y))_{I^{*}}$ be the restriction of $x^{*}_{\lambda}(y)$ to its support and $S^{*} = \sign \left( ( x^{*}_{\lambda}(y) )_{I^{*}} \right)$. From Lemma~\ref{lem3} we have,
\[
\widehat{\mu}_{\lambda}(y) = A x^{*}_{\lambda}(y) = A_{I^{*}} ( x^{*}_{\lambda}(y) )_{I^{*}}.
\]
According to Lemma~\ref{lem1}, we know that
\begin{eqnarray*}
&& A_{I^{*}}^{\Tr} ( y - \widehat{\mu}_{\lambda}(y) ) = \lambda S^{*};\\
&& \vert \langle a_{k}, y - \widehat{\mu}_{\lambda}(y) \rangle \vert \leq \lambda, \forall~ k \in (I^{*})^{c}.
\end{eqnarray*}
Furthermore, from \eqref{eq11}, we get the following implicit form of $x^{*}_{\lambda}(y)$
\begin{equation}
( x^{*}_{\lambda}(y) )_{I^{*}} = A_{I^{*}}^{+} y - \lambda ( A_{I^{*}}^{\Tr} A_{I^{*}} )^{-1} S^*.
\end{equation}
It follows that
\begin{equation}\label{y}
\widehat{\mu}_{\lambda}(y) = P_{V_{I^{*}}} (y) - \lambda d_{I^{*},S^{*}},
\end{equation}
and
\begin{equation}\label{r}
\widehat{r}_{\lambda} (y) = y - \widehat{\mu}_{\lambda}(y) = P_{V_{I^{*}}^{\perp}} (y) + \lambda d_{I^{*},S^{*}} ~,
%P_{V_{I}^{\perp}}
\end{equation}
where $d_{I^{*},S^{*}} = (A_{I^{*}}^{+})^{\Tr} S^{*}$.
%where $V_{I^{*}} = \mbox{span} ( a_{i} )_{i \in I^{*}}$, $P_{V_{I^{*}}} = A_{I^{*}} A_{I^{*}}^{+}$ is the orthogonal projection onto $V_{I}$, $P_{V_{I^{*}}^{\perp}} = \mathrm{Id}_{n \times n} - P_{V_{I^{*}}}$ is the orthogonal projection onto the orthogonal complement $V_{I^{*}}^{\perp}$ of $V_{I^{*}}$, and $d_{I^{*},S^{*}} = (A_{I^{*}}^{+})^{\Tr} S^{*}$.\\
We define the following set of indices
\begin{equation}\label{j}
J = \lbrace  j : \vert \langle a_{j}, \widehat{r}_{\lambda} (y)  \rangle \vert = \lambda \rbrace.
\end{equation}
%Thus, we note that $J$ is the set of supports of all solutions of the
%Lasso problem \eqref{eq1}. 
From Lemma~\ref{lem1} we deduce that 
%In particular, we have
\[
I^{*} \subset J.
\]
%Furthermore, from Lemma~\ref{lem2}, if $I^{*} = J$ then
%$x^{*}_{\lambda}(y)$ is the unique solution of \eqref{eq1}. 
Since the orthogonal projection is a self-adjoint operator and from \eqref{r}, for all $j \in J$, we have
\begin{equation}\label{c}
\vert \langle P_{V_{I^{*}}^{\perp}} (a_{j}), y \rangle + \lambda \langle a_{j}, d_{I^{*},S^{*}} \rangle  \vert = \lambda.
\end{equation}
As $y \in G_{\lambda}$, we deduce that if $j \in J \cap (I^{*})^{c}$ then inevitably we have
\begin{equation}\label{f-r}
a_{j} \in V_{I^{*}},~ \mbox{and therefore}~ \vert \langle a_{j}, d_{I^{*},S^{*}} \rangle \vert = 1.
\end{equation}
In fact, if $a_j \not \in V_{I^{*}}$ then $( I^{*},j, S^{*} ) \in \Omega$ and from \eqref{c} we have that $y \in H_{I^{*},j,S^{*}}$, which is a contradiction with $y \in G_{\lambda}$.\\
Therefore, the collection of vectors $(a_{i})_{i \in I^{*}}$ forms a basis of $V_{J} = \span (a_{j})_{j \in J}$. Now, suppose that $\widehat{x}_{\lambda}(y)$ is another solution of \eqref{eq1}, such that its support $I$ is different from $I^{*}$. If $A_{I}$ is full column rank, then by using the same above arguments we can deduce that $(a_{i})_{i \in I}$ forms also a basis of $V_{J}$. Therefore, we have
\[
\vert I \vert = \vert I^{*} \vert = \dim (V_{J}).
\]
On the other hand, if $A_{I}$ is not full rank, then there exists a subset $I_{0} \subsetneq I$ such that $A_{I_{0}}$ is full rank (see the discussion following Theorem~\ref{theo-general}) and $(a_{i})_{i \in I_{0}}$ forms also a basis of $V_{J}$, which implies that
\[
\vert I \vert > \vert I_0\vert=\dim (V_{J}) = \vert I^{*} \vert.
\]
We conclude that for any solution $\widehat{x}_\lambda(y)$ of \eqref{eq1}, we have
\[
\vert \supp ( \widehat{x}_\lambda(y) ) \vert \geq \vert I^{*} \vert,
\]
and then $\vert I^{*} \vert$ is equal to the minimum of the cardinalities of the supports of solutions of \eqref{eq1}. This proves the first part of the theorem.\\

Let's turn to the second statement. Note that $G_{\lambda}$ is an open set and all components of $ (
x^{*}_{\lambda}(y) )_{I^{*}} $ are nonzero, so we can choose a small
enough $\varepsilon$ such that $\ball (y, \varepsilon) \subsetneq
G_{\lambda}$, that is, for all $z \in \ball (y, \varepsilon)$, $z\in G_{\lambda}$.
Now, let $x^{1}_{\lambda}(z)$ be the vector supported in $I^{*}$ and defined by
\begin{equation}\label{autre}
( x^{1}_{\lambda}(z) )_{I^{*}} = A_{I^{*}}^{+} z - \lambda
(A_{I^*}^\Tr A_{I^*})^{-1}S^{*}
=(x^*_{\lambda}(y))_{I^*}+A_{I^{*}}^{+}(z-y).
\end{equation}
If $\varepsilon$ is small enough, then for all  $z \in \ball (y,
\varepsilon)$, we have
\begin{equation}
\sign( x^{1}_{\lambda}(z) )_{I^{*}} = \sign ( x^{*}_{\lambda}(y)
)_{I^{*}}= S^*.
\end{equation}

In the rest of the proof, we invoke Lemma~\ref{lem1} to show that, for $\varepsilon$ small enough, 
$x^1_{\lambda}(z)$ is actually a solution of $(\mbox{P}_{1} ( z, \lambda ))$. 
First we notice that $z-Ax^1_{\lambda}(z)=P_{V_I^{\perp}}(z)+\lambda
d_{I^*,S^*}$.
It follows that 
\begin{equation}
A_{I^*}^\Tr(z-Ax^1_{\lambda}(z))=\lambda A_{I^*}^\Tr d_{I^*,S^*}=\lambda
S^*=\lambda \sign{( x^{1}_{\lambda}(z) )_{I^{*}}}.
\end{equation}
Moreover for all $j\in J\cap I^*$, from \eqref{f-r}, we have that
\begin{eqnarray*}
\vert \langle a_{j}, z - A x^{1}_{\lambda}(z) \rangle \vert &=& \vert \langle a_{j}, P_{V_{I^{*}}^{\perp}} (z) + \lambda d_{I^{*},S^{*}} \rangle \vert\\ &=& \vert \langle P_{V_{I^{*}}^{\perp}} (a_{j}), z \rangle + \lambda \langle a_{j}, d_{I^{*},S^{*}} \rangle \vert\\ &=& \lambda \vert \langle a_{j}, d_{I^{*},S^{*}} \rangle \vert = \lambda.
\end{eqnarray*}
and for all $j\notin J$
\begin{equation*}
\vert \langle a_j,z-Ax^1_{\lambda}(z)\rangle\vert\leq 
\vert \langle a_j,y-Ax^*_{\lambda}(y)\rangle\vert+\vert \langle P_{V_{I^{*}}^{\perp}} (a_{j}),z-y\rangle\vert
\end{equation*}
Since for all $j\notin J$, $\vert \langle
a_j,y-Ax^*_{\lambda}\rangle\vert<\lambda$, there exists $\varepsilon$
such that for all $z\in \ball (y, \varepsilon)$ and $\forall~ j\notin J$, we have
\[
\vert \langle a_j,z-Ax^1_{\lambda}(z)\rangle\vert<\lambda.
\]
% Since 
% verifies the following conditions:
% \begin{eqnarray}
% && \sign \left( ( x^{1}_{\lambda}(z) )_{I^{*}} \right) = S^{*},\\
% && \vert \langle a_{k}, z - A x^{1}_{\lambda}(z) \rangle \vert < \lambda, \forall~ k \in J^{c}.
% \end{eqnarray}
% Our aim is to proof that $x^{1}_{\lambda}(z)$ is a solution of $(\mbox{P}_{1} ( z, \lambda ))$. First, by definition, we have 
% \[
% A_{I^{*}}^{\Tr} ( z - A x^{1}_{\lambda}(z) ) = \lambda S^{*}. 
% \]
% Let $j \in J \cap (I^{*})^{c}$. Thus, from \eqref{f-r}, we have that
% \begin{eqnarray*}
% \vert \langle a_{j}, z - A x^{1}_{\lambda}(z) \rangle \vert &=& \vert \langle a_{j}, P_{V_{I^{*}}^{\perp}} (z) + \lambda d_{I^{*},S^{*}} \rangle \vert\\ &=& \vert \langle P_{V_{I^{*}}^{\perp}} (a_{j}), z \rangle + \lambda \langle a_{j}, d_{I^{*},S^{*}} \rangle \vert\\ &=& \lambda \vert \langle a_{j}, d_{I^{*},S^{*}} \rangle \vert = \lambda.
% \end{eqnarray*}
Therefore, we obtain
\[
\vert \langle a_{j}, z - A x^{1}_{\lambda}(z) \rangle \vert \leq \lambda, \forall~ j \in (I^{*})^{c}.
\]
Which, by Lemma~\ref{lem1}, means that $x^{1}_{\lambda}(z)$ is a solution of $(\mbox{P}_{1} ( z, \lambda ))$, and the unique  Lasso response associated to $(\mbox{P}_{1} ( z, \lambda ))$, denoted by $\widehat{\mu}_{\lambda}(z)$, is defined by
\begin{equation}\label{z}
\widehat{\mu}_{\lambda}(z) = P_{V_{I^{*}}} (z) - \lambda d_{I^{*},S^{*}}.
\end{equation}
Therefore, from \eqref{y} and \eqref{z}, we can deduce that for all $z \in \ball (y, \varepsilon)$ we have
\[
\widehat{\mu}_{\lambda}(z) = \widehat{\mu}_{\lambda}(y) + P_{V_{I^{*}}} (z - y).
\]
\end{proof}
\begin{proof}[Proof of Corollary~\ref{Corollary1}]
We showed that there exists $\varepsilon$ sufficiently small such that 
\begin{equation}
\Vert z - y \Vert_{2} \leq \varepsilon \Rightarrow \widehat{\mu}_{\lambda}(z) = \widehat{\mu}_{\lambda}(y) + P_{V_{I^{*}}} (z - y).
\end{equation}
Let $h \in V_{I^{*}}$ such that $\Vert h \Vert_{2} \leq \varepsilon$ and $z = y + h$. Thus, we have that $\Vert z - y \Vert_{2} \leq \varepsilon$ and then
\begin{equation}
\Vert \widehat{\mu}_{\lambda}(z) - \widehat{\mu}_{\lambda}(y) \Vert_{2} = \Vert P_{V_{I^{*}}} (h) \Vert_{2} = \Vert h \Vert_{2} \leq \varepsilon.
\end{equation}
Therefore, the Lasso response $\widehat{\mu}_{\lambda}(y)$ is uniformly Lipschitz on $G_{\lambda}$. Moreover, $\widehat{\mu}_{\lambda}(y)$ is a continuous function of $y$, and thus $\widehat{\mu}_{\lambda}(y)$ is
uniformly Lipschitz on $\R^n$. Hence, $\widehat{\mu}_{\lambda}(y)$ is almost differentiable; see \cite{meyer} and \cite{efronn}.

On the other hand, we proved that there exists a neighborhood of $y$, such that for all $z$ in this neighborhood, there exists a solution of the Lasso problem $(\mbox{P}_{1} ( z, \lambda ))$, which has the same support and the same sign as $x^{*}_{\lambda}(y)$, and thus $\widehat{\mu}_{\lambda}(z)$ belongs to the vector space $V_{I^{*}}$, whose dimension equals to $\vert I^{*} \vert$, see \eqref{y} and \eqref{z}. Therefore, $\widehat{\mu}_{\lambda}(y)$ is a locally affine function of $y$, and then
\begin{equation}
J_{\widehat{\mu}_{\lambda}(y)} = P_{V_{I^{*}}} ~.
\end{equation}
Then the trace formula \eqref{divv} implies that
\begin{equation}\label{divvv}
\divg \left( \widehat{\mu}_{\lambda}(y) \right) = \tr \left( P_{V_{I^{*}}} \right) = \vert I^{*} \vert.
\end{equation}
This holds almost everywhere since $G_{\lambda}$ is of full measure, and \eqref{dfff} is obtained by invoking Stein's lemma.
\end{proof}

\begin{proof}[Proof of Theorem~\ref{th2}]
First, consider the following random variable
\[
Q_{1} ( \widehat{\mu}_{\lambda}(y)) = \Vert \widehat{\mu}_{\lambda}(y) \Vert_{2}^{2} + \Vert \mu \Vert_{2}^{2} - 2 \langle y, \widehat{\mu}_{\lambda}(y) \rangle + 2 \sigma^{2} \divg (\widehat{\mu}_{\lambda}(y)).
\]
From Stein's lemma, we have 
\[
\E \langle \varepsilon, \widehat{\mu}_{\lambda}(y) \rangle = \sigma^{2} \E \parenth{\divg (\widehat{\mu}_{\lambda}(y))}.
\]
Thus, we can deduce that $Q_{1} ( \widehat{\mu}_{\lambda}(y))$ and $\sure ( \widehat{\mu}_{\lambda}(y))$ 
are unbiased estimator of the prediction risk, i.e.
\[
\E \parenth{\sure ( \widehat{\mu}_{\lambda}(y))} = \E \parenth{Q_{1} ( \widehat{\mu}_{\lambda}(y))} = \E \parenth{\se(\widehat{\mu}_{\lambda}(y))} = \mse (\mu).
\]
Moreover, note that $\sure ( \widehat{\mu}_{\lambda}(y)) - Q_{1} ( \widehat{\mu}_{\lambda}(y)) = \Vert y \Vert_{2}^{2} - \E \parenth{ \Vert y \Vert_{2}^{2} }$, where
\begin{equation}\label{p0}
\E \left( \Vert y \Vert_{2}^{2} \right) = n \sigma^{2} + \Vert \mu \Vert_{2}^{2},~\mbox{and}~ \V \parenth{ \Vert y \Vert_{2}^{2} } = 2 \sigma^{4} \left( n + 2 \dfrac{\Vert \mu \Vert_{2}^{2}}{\sigma^{2}} \right).
\end{equation}
Now, we remark also that
\begin{eqnarray}
Q_{1} ( \widehat{\mu}_{\lambda}(y)) - \se(\widehat{\mu}_{\lambda}(y)) = 2 \left( \sigma^{2} \divg (\widehat{\mu}_{\lambda}(y)) - \langle \varepsilon, \widehat{\mu}_{\lambda}(y)\rangle  \right).
\end{eqnarray}
After an elementary calculation, we obtain
\begin{equation}\label{p1}
\E ( \sure ( \widehat{\mu}_{\lambda}(y)) - \se(\widehat{\mu}_{\lambda}(y)) )^{2} = \E ( Q_{1} ( \widehat{\mu}_{\lambda}(y)) - \se(\widehat{\mu}_{\lambda}(y)) )^{2} + \V \parenth{ \Vert y \Vert_{2}^{2} } + 4 T,
\end{equation}
where
\begin{equation}
T = \sigma^{2} \E \parenth{\divg (\widehat{\mu}_{\lambda}(y)) \Vert y \Vert_{2}^{2}} - \E \parenth{\langle \varepsilon, \widehat{\mu}_{\lambda}(y)\rangle \Vert y \Vert_{2}^{2}} = T_{1} + T_{2},
\end{equation}
with
\begin{equation}
T_{1} = 2 \parenth{ \sigma^{2} \E \parenth{\divg (\widehat{\mu}_{\lambda}(y)) \langle \varepsilon, \mu \rangle} - \E \parenth{\langle \varepsilon, \widehat{\mu}_{\lambda}(y)\rangle \langle \varepsilon, \mu \rangle} }
\end{equation}
and
\begin{equation}
T_{2} = \sigma^{2} \E \parenth{\divg ( \widehat{\mu}_{\lambda}(y)) \Vert \varepsilon \Vert_{2}^{2}} - \E \parenth{\langle \varepsilon, \widehat{\mu}_{\lambda}(y)\rangle \Vert \varepsilon \Vert_{2}^{2}}.
\end{equation}
%Now, recall that the expression of Lasso response $\widehat{\mu}_{\lambda}$ is given by \eqref{y}. Thus, since $( \varepsilon_{i} )$ is a sequence of i.i.d. centered Gaussian random variables with variance $\sigma^{2} > 0$, we have
%\begin{equation}\label{ass}
%\left[ \left( \widehat{\mu}_{\lambda}(y)\right)_{i} \varepsilon_{i} f({\varepsilon_{i}}) \right]_{-\infty}^{+\infty} \rightarrow 0,~ \mbox{and}~ \left[ \left( \widehat{\mu}_{\lambda}(y)\right)_{i} \varepsilon_{i}^{2} f({\varepsilon_{i}}) \right]_{-\infty}^{+\infty} \rightarrow 0,~ \forall~ i = 1, \cdots, n,
%\end{equation}
%where $f({\varepsilon_{i}})$ is the gaussian probability density of $\varepsilon_{i}$. Hence, by using the fact that $f ({\varepsilon_{i}})$ satisfies $\varepsilon_{i} f ({\varepsilon_{i}}) = - \sigma^{2} f^{'} ({\varepsilon_{i}})$ and integrations by parts with \eqref{ass}, we find that
Hence, by using the fact that a Gaussian probability density $\varphi ({\varepsilon_{i}})$ satisfies $\varepsilon_{i} \varphi ({\varepsilon_{i}}) = - \sigma^{2} \varphi^\prime ({\varepsilon_{i}})$ and integrations by parts, we find that
%\[
%\E \langle \varepsilon, \widehat{\mu}_{\lambda}(y)\rangle \langle \varepsilon, \mu \rangle = \sigma^{2} \E  \divg ( \widehat{\mu}_{\lambda}(y)) \langle \varepsilon, \mu \rangle + \sigma^{2} \E \langle \widehat{\mu}_{\lambda}, \mu \rangle
%\]
\[
T_{1} = -2 \sigma^{2} \E \parenth{\langle \widehat{\mu}_{\lambda}, \mu \rangle}
\]
and
%\[
%\E \langle \varepsilon, \widehat{\mu}_{\lambda}(y)\rangle \Vert \varepsilon \Vert_{2}^{2} = \sigma^{2} \E \divg (\widehat{\mu}_{\lambda}(y)) \Vert \varepsilon \Vert_{2}^{2} + 2 \sigma^{4} \E \left[ \divg ( \widehat{\mu}_{\lambda}(y)) \right].
%\]
\[
T_{2} = -2 \sigma^{4} \E \parenth{ \divg ( \widehat{\mu}_{\lambda}(y)) }.
\]
It follows that
\begin{equation}\label{p2}
T = - 2 \sigma^{2} \big( \E \parenth{\langle \widehat{\mu}_{\lambda}, \mu \rangle} + \sigma^{2} \E \parenth{\divg (\widehat{\mu}_{\lambda}(y))} \big).
\end{equation}
Moreover, from \cite[Property~1]{luisier}, we know that
\begin{equation}
\E ( Q_{1} ( \widehat{\mu}_{\lambda}(y)) - \se(\widehat{\mu}_{\lambda}(y)) )^{2} = 4 \sigma^{2} \bigg( \E \parenth{ \Vert \widehat{\mu}_{\lambda}(y)\Vert_{2}^{2} } + \sigma^{2} \E \parenth{ \tr \parenth{ \left( J_{\widehat{\mu}_{\lambda}(y)} \right)^{2} } } \bigg),
\end{equation}
Thus, since $J_{\widehat{\mu}_{\lambda}(y)} = P_{V_{I^{*}}}$ which is an orthogonal projector (hence self-adjoint and idempotent), we have $\tr \parenth{\left( J_{\widehat{\mu}_{\lambda}(y)} \right)^2} = \divg ( \widehat{\mu}_{\lambda}(y)) = \vert I^{*} \vert$. Therefore, we get
\begin{equation}\label{p3}
\E ( Q_{1} ( \widehat{\mu}_{\lambda}(y)) - \se(\widehat{\mu}_{\lambda}(y)) )^{2} = 4 \sigma^{2} \left( \E \left( \Vert \widehat{\mu}_{\lambda}(y)\Vert_{2}^{2} \right) + \sigma^{2} \E \left( \vert I^{*} \vert \right) \right).
\end{equation}
Furthermore, observe that
\begin{equation}\label{p4}
\E \parenth{\sure (\widehat{\mu}_{\lambda}(y))} = -n \sigma^{2} + \E \left( \Vert \widehat{\mu}_{\lambda}(y)- y \Vert_{2}^{2} \right) + 2 \sigma^{2} \E \left( \vert I^{*} \vert \right) . % \right).
\end{equation}
Therefore, by combining \eqref{p0}, \eqref{p1}, \eqref{p2} and \eqref{p3}, we obtain
\begin{eqnarray*}
\E ( \sure ( \widehat{\mu}_{\lambda}(y)) - \se(\widehat{\mu}_{\lambda}(y)) )^{2} &=& 2 n \sigma^{4} + 4 \sigma^{2} \E \parenth{\se(\widehat{\mu}_{\lambda}(y))} - 4 \sigma^{4} \E \left( \vert I^{*} \vert \right)\\ &=&  2 n \sigma^{4} + 4 \sigma^{2} \E \left( \sure ( \widehat{\mu}_{\lambda}(y)) \right) - 4 \sigma^{4} \E \left( \vert I^{*} \vert \right)\\(\mbox{by using}~ \eqref{p4})&=& - 2 n \sigma^{4} + 4 \sigma^{2} \E \left( \Vert \widehat{\mu}_{\lambda}(y)- y \Vert_{2}^{2} \right) + 4 \sigma^{4} \E \left( \vert I^{*} \vert \right).
\end{eqnarray*}

On the other hand, since $x^{*}_{\lambda}(y)$ is a minimizer of the Lasso problem \eqref{eq1}, we observe that
\begin{eqnarray*}
\frac{1}{2} \Vert \widehat{\mu}_{\lambda}(y)- y \Vert_{2}^{2} \leq \frac{1}{2} \Vert \widehat{\mu}_{\lambda}(y)- y \Vert_{2}^{2} + \lambda \Vert x^{*}_{\lambda}(y) \Vert_{1} \leq \frac{1}{2} \Vert A . 0 - y \Vert_{2}^{2} + \lambda \Vert 0\Vert_{1} = \frac{1}{2} \Vert y \Vert_{2}^{2}.
\end{eqnarray*}
Therefore, we have
\begin{equation}\label{p5}
\E \left( \Vert \widehat{\mu}_{\lambda}(y)- y \Vert_{2}^{2} \right) \leq \E \left( \Vert y \Vert_{2}^{2} \right) =  n \sigma^{2} + \Vert \mu \Vert_{2}^{2}.
\end{equation}
Then, since \added{$\vert I^{*} \vert = \bigO{n}$} and from \eqref{p5}, we have
\begin{equation}\label{bs}
\E \left( \left( \dfrac{\sure ( \widehat{\mu}_{\lambda}(y)) - \se(\widehat{\mu}_{\lambda}(y))}{n \sigma^{2}} \right)^{2} \right) \leq \dfrac{6}{n} + \dfrac{4 \Vert \mu \Vert_{2}^{2}}{n^{2} \sigma^{2}}.
\end{equation}
Finally, since $\Vert \mu \Vert_{2} < +\infty$, we can deduce that
\[
\E \left( \left( \dfrac{\sure ( \widehat{\mu}_{\lambda}(y)) - \se(\widehat{\mu}_{\lambda}(y))}{n \sigma^{2}} \right)^{2} \right) = \bigO{\frac{1}{n}}.
\]
\end{proof}
\section{Discussion}\label{group}
In this paper we proved that the number of nonzero coefficients of a particular solution of the Lasso problem is an unbiased estimate of the degrees of freedom of the Lasso response for linear regression models. This result covers both the over and underdetermined cases. This was achieved through a divergence formula, valid almost everywhere except on a set of measure zero. We gave a precise characterization of this set, and the latter turns out to be larger than the set of all the vectors associated to the transition points considered in \cite{zou} in the overdetermined case. We also highlight the fact that even in the overdetermined case, the set of transition points is not sufficient for the divergence formula to hold.
%hence providing a counterexample to some of the key results in \cite{zou}. 
%In the latter case, no additional assumptions are needed to ensure the uniqueness of the solution, in order to achieve this result.

%Since \cite{zou}, only a few papers have appeared on the degrees of freedom of the generalized Lasso problem, we cite in particular \cite{kato} and \cite{tib11}. These works also focus on the overdetermined case.
We think that some techniques developed in this article can be applied to derive the degrees of freedom of other nonlinear estimating procedures. Typically, a natural extension of this work is to consider other penalties such as those promoting structured sparsity, e.g. the group Lasso.

\paragraph*{Acknowledgement}
This work was partly funded by the ANR grant NatImages, ANR-08-EMER-009.


\begin{thebibliography}{xx}
\bibitem{akaike}Akaike, H. (1973). Information theory and an extension of the maximum likelihood principle. Second International Symposium on Information Theory 267-281.
%\bibitem{bach}Bach, F. (2008). Consistency of the group Lasso and multiple kernel learning. Journal of Machine Learning Research, 9:1179-1225.
%\bibitem{cai}Cai T.T. and Silverman B.W. (2001). Incorporating information on neighbouring coefficients into wavelet estimation, Sankhya: The Indian Journal of Statistics 63(Series B, Pt. 2), 127–148.
\bibitem{bickel09}Bickel, P. J., Ritov, Y., and Tsybakov, A., (2009). Simultaneous analysis of Lasso and Dantzig selector. Annals of Statistics. 37 1705?1732.
\bibitem{craven}Craven, P. and Wahba, G. (1979). Smoothing Noisy Data with Spline Functions: 
estimating the correct degree of smoothing by the method of generalized cross validation. Numerische Mathematik 31, 377-403.
\bibitem{daubechies}Daubechies, I., Defrise, M., and Mol, C. D. (2004). An iterative thresholding algorithm for linear inverse problems with a sparsity constraint, Communications on Pure and Applied Mathematics 57, 1413-1541.
\bibitem{dossal}Dossal, C (2007). A necessary and sufficient condition for exact recovery by l1 minimization. Technical report, HAL-00164738:1.
\bibitem{ef}Efron, B. (2004). The estimation of prediction error: Covariance penalties and cross-validation
(with discussion). J. Amer. Statist. Assoc. 99 619-642.
\bibitem{efronn}Efron, B., Hastie, T., Johnstone, I. and Tibshirani, R. (2004). Least angle
regression (with discussion). Ann. Statist. 32 407-499.
\bibitem{efron}Efron, B. (1981). How biased is the apparent error rate of a prediction rule. J. Amer. Statist. Assoc. vol. 81 pp. 461-470.
\bibitem{fanetli}Fan, J. and Li, R. (2001). Variable selection via nonconcave penalized likelihood and
its oracle properties. J. Amer. Statist. Assoc. 96 1348-1360.
\bibitem{fan}Fan, J. and Peng, H. (2004). Nonconcave penalized likelihood with a diverging number of parameters. The Annals of Statistics, 32(3), 928-961.
%\bibitem{fletcher}Fletcher, R. Practical Methods of Optimization. John Wiley and Sons, Inc., 2nd edition, 1987.
\bibitem{fuchs}Fuchs, J. J. (2004). On sparse representations in arbitrary redundant bases.
IEEE Trans. Inform. Theory, vol. 50, no. 6, pp. 1341-1344.
\bibitem{kato09}Kato, K. (2009). On the degrees of freedom in shrinkage estimation. Journal of Multivariate Analysis 100(7), 1338-1352.
\bibitem{luisier}Luisier, F. (2009). The $\sure$-LET approach to image denoising. Ph.D. dissertation, EPFL, Lausanne. Available: http://library.epfl.ch/theses/?nr=4566.
\bibitem{mallows}Mallows, C. (1973). Some comments on $C_{p}$. Technometrics 15, 661-675.
\bibitem{meyer}Meyer, M. and Woodroofe, M. (2000). On the degrees of freedom in shape restricted
regression. Ann. Statist. 28 1083-1104
\bibitem{nardi}Nardi, Y. and Rinaldo, A (2008). On the asymptotic properties of the group Lasso estimator for linear models. Electronic Journal of Statistics, 2 605-633.
\bibitem{osborne1}Osborne, M., Presnell, B. and Turlach, B. (2000a). A new approach to variable
selection in least squares problems. IMA J. Numer. Anal. 20 389-403.
\bibitem{osborne2}Osborne, M. R., Presnell, B. and Turlach, B. (2000b). On the LASSO and its dual.
J. Comput. Graph. Statist. 9 319-337.
\bibitem{ravi}Ravikumar, P., Liu, H., Lafferty, J., and Wasserman, L (2008). Spam: Sparse additive models. In Advances in Neural Information Processing Systems (NIPS), volume 22.
\bibitem{rosset04}Rosset, S., Zhu, J., Hastie, T. (2004). Boosting as a Regularized Path
to a Maximum Margin Classifier. J. Mach. Learn. Res. 5 941-973.
\bibitem{sardy}Sardy, S., Bruce, A., and Tseng, P. (2000). Block coordinate relaxation methods for nonparametric wavelet denoising. J. of Comp. Graph. Stat. 9(2) 361?379.
\bibitem{schwarz}Schwarz, G. (1978). Estimating the dimension of a model. Ann. Statist. 6 461-464.
\bibitem{stein}Stein, C. (1981). Estimation of the mean of a multivariate normal distribution. Ann.
Statist. 9 1135-1151.
\bibitem{tib11}Tibshirani, R. and Taylor, J. (2011). The Solution Path of the Generalized Lasso. Annals of Statistics. In Press.
\bibitem{tib12}Tibshirani, R. and Taylor, J. (2012). Degrees of Freedom in Lasso Problems. Technical report, arXiv:1111.0653.
\bibitem{tibshirani}Tibshirani, R. (1996). Regression shrinkage and selection via the Lasso. J. Roy.
Statist. Soc. Ser. B 58(1) 267-288.
\bibitem{tropp}Tropp J. A. (2006). Just relax: convex programming methods for identifying sparse signals
in noise, IEEE Trans. Info. Theory 52 (3), 1030-1051.
\bibitem{vaiter11}Vaiter, S., Peyr\'e, G., Dossal, C. and Fadili, M.J. (2011), Robust sparse analysis regularization. arXiv:1109.6222.
\bibitem{yuan} Yuan, M. and Lin, Y. (2006). Model selection and estimation in regression with grouped
variables. J. Roy. Statist. Soc. Ser. B 68 49-67.
\bibitem{zhao}Zhao, P. and Bin, Y. (2006). On model selection consistency of Lasso. Journal of Machine Learning Research, 7, 2541-2563.
\bibitem{zouu}Zou, H. (2006). The adaptive Lasso and its oracle properties. Journal of the American Statistical Association, 101(476), 1418-1429
\bibitem{zou}Zou, H., Hastie, T. and Tibshirani, R. (2007). On the "degrees of freedom" of the
Lasso. Ann. Statist. Vol. 35, No. 5. 2173-2192.
\end{thebibliography}
\end{document}